\documentclass[a4paper]{amsart}

\usepackage{hyperref}

\usepackage[initials,nobysame,alphabetic,msc-links]{amsrefs}
\DefineSimpleKey{bib}{how}

\renewcommand{\PrintDOI}[1]{%
    \href{http://dx.doi.org/#1}{{\tt DOI:#1}}%
}
\renewcommand{\eprint}[1]{#1}
\BibSpec{book}{%
    +{}  {\PrintPrimary}                {transition}
    +{.} { \PrintDate}                  {date}
    +{.} { \textit}                     {title}
    +{.} { }                            {part}
    +{:} { \textit}                     {subtitle}
    +{,} { \PrintEdition}               {edition}
    +{}  { \PrintEditorsB}              {editor}
    +{,} { \PrintTranslatorsC}          {translator}
    +{,} { \PrintContributions}         {contribution}
    +{,} { }                            {series}
    +{,} { \voltext}                    {volume}
    +{,} { }                            {publisher}
    +{,} { }                            {organization}
    +{,} { }                            {address}
    +{,} { }                            {status}
    +{,} { \PrintDOI}                   {doi}
    +{,} { \PrintISBNs}                 {isbn}
    +{}  { \parenthesize}               {language}
    +{}  { \PrintTranslation}           {translation}
    +{;} { \PrintReprint}               {reprint}
    +{.} { }                            {note}
    +{.} {}                             {transition}
    +{}  {\SentenceSpace \PrintReviews} {review}
}
\BibSpec{article}{%
    +{}  {\PrintAuthors}                {author}
    +{,} { \textit}                     {title}
    +{.} { }                            {part}
    +{:} { \textit}                     {subtitle}
    +{,} { \PrintContributions}         {contribution}
    +{.} { \PrintPartials}              {partial}
    +{,} { }                            {journal}
    +{}  { \textbf}                     {volume}
    +{}  { \PrintDatePV}                {date}
    +{,} { \issuetext}                  {number}
    +{,} { \eprintpages}                {pages}
    +{,} { }                            {status}
    +{,} { \PrintDOI}                   {doi}
    +{,} { \eprint}        {eprint}
    +{}  { \parenthesize}               {language}
    +{}  { \PrintTranslation}           {translation}
    +{;} { \PrintReprint}               {reprint}
    +{.} { }                            {note}
    +{.} {}                             {transition}
    +{}  {\SentenceSpace \PrintReviews} {review}
}
\BibSpec{collection.article}{%
    +{}  {\PrintAuthors}                {author}
    +{,} { \textit}                     {title}
    +{.} { }                            {part}
    +{:} { \textit}                     {subtitle}
    +{,} { \PrintContributions}         {contribution}
    +{,} { \PrintConference}            {conference}
    +{}  {\PrintBook}                   {book}
    +{,} { }                            {booktitle}
    +{,} { \PrintDateB}                 {date}
    +{,} { pp.~}                        {pages}
    +{,} { }                            {publisher}
    +{,} { }                            {organization}
    +{,} { }                            {address}
    +{,} { }                            {status}
    +{,} { \PrintDOI}                   {doi}
    +{,} { \eprint}        {eprint}
    +{}  { \parenthesize}               {language}
    +{}  { \PrintTranslation}           {translation}
    +{;} { \PrintReprint}               {reprint}
    +{.} { }                            {note}
    +{.} {}                             {transition}
    +{}  {\SentenceSpace \PrintReviews} {review}
}
\BibSpec{misc}{%
  +{}{\PrintAuthors}  {author}
  +{,}{ \textit}      {title}
  +{.}{ }             {how}
  +{}{ \parenthesize} {date}
  +{,} { available at \eprint}        {eprint}
  +{,}{ available at \url}{url}
  +{,}{ }             {note}
  +{.}{}              {transition}
}

\usepackage{amssymb}
\usepackage{amscd}  
\theoremstyle{plain}
\newtheorem{thm}[]{Theorem}
\newtheorem{prop}[]{Proposition}
\newtheorem{lem}[prop]{Lemma}
\newtheorem{cor}[prop]{Corollary}

\theoremstyle{definition}
\newtheorem{dfn}[prop]{Definition}

\theoremstyle{remark}
\newtheorem{rmk}[prop]{Remark}
\newtheorem{example}[prop]{Example}

\newcommand{\ensemble}[1]{\left\{ #1 \right\}}
\newcommand{\suchthat}{\mid}
\newcommand{\norm}[1]{\left\| #1 \right\|}
\newcommand{\absolute}[1]{\left| #1 \right|}

\newcommand{\cptops}{\mathcal{K}}

\newcommand{\Z}{\mathbb{Z}}

\newcommand{\C}{\mathbb{C}}
\newcommand{\R}{\mathbb{R}}
\newcommand{\T}{\mathbb{T}}
\newcommand{\fin}{\text{fin}}

\newcommand{\even}{\text{even}}
\newcommand{\smooth}[1]{\mathcal{#1}}
\newcommand{\U}{\mathrm{U}}
\newcommand{\Adj}{\mathrm{Ad}}
\newcommand{\tildAdj}{\widetilde{\Adj}}
\newcommand{\Double}{\mathsf{D}}

\newcommand{\topol}{\text{top}}
\newcommand{\HMEnd}[1]{\mathcal{L}\left ( #1 \right )}
\newcommand{\MA}{\mathcal{M}}

\newcommand{\bdl}[1]{\mathcal{#1}}
\newcommand{\bimod}[1]{\mathcal{#1}}

\DeclareMathOperator{\maxten}{\otimes_{\max}}

\DeclareMathOperator{\Res}{Res}
\DeclareMathOperator{\Ind}{Ind}
\DeclareMathOperator{\dom}{dom}

\DeclareMathOperator{\ev}{ev}
\DeclareMathOperator{\Id}{Id}

\DeclareMathOperator{\mintensor}{\otimes_{\min}}

\DeclareMathOperator{\KK}{\mathit{KK}}
\DeclareMathOperator{\RKK}{\mathcal{R}\mathit{KK}}

\begin{document}

\title{Deformation of algebras associated with group cocycles}
\author{Makoto Yamashita}
\date{minor changes August 3, 2023; revised June 6, 2013; October 15, 2011} 
\address{Dipartimento di Matematica,
Universit\`{a} degli Studi di Roma ``Tor Vergata''\\
Via della Ricerca Scientifica 1, 00133 Rome, Italy}
\curraddr{Department of Mathematics, University of Oslo}
\email{makotoy@math.uio.no}


\begin{abstract}
 We study deformation of algebras with coaction symmetry of reduced algebra of discrete groups, where the deformation parameter is given continuous family of group $2$-cocycles.  When the group satisfies the Baum--Connes conjecture with coefficients, we obtain isomorphism of K-groups of the deformed algebras.  This extends both the $\theta$-deformation of Rieffel on $\T^n$-actions, and a recent result of Echterhoff, L\"{u}ck, Phillips, and Walters on the K-groups on the twisted group algebras.
\end{abstract}

\maketitle

\section{Introduction}
\label{sect:intro}

Deformation of algebras has been an important technique to produce interesting examples in the study of operator algebras and noncommutative geometry.  One of the most famous examples is the noncommutative torus algebra $C(\T^2_\theta)$, which is universally generated by two unitaries $u$ and $v$ satisfying $u v = e^{i \theta} v u$ for a given real parameter $\theta$.  Since the case of $\theta \in 2 \pi \Z$ gives the algebra of the algebra of continuous functions on the usual $2$-torus, one may think of $C(\T^2_\theta)$ as an algebra representing a deformed object of the $2$-torus.  This example inspired many interesting ideas in the early development of noncommutative geometry by Connes~\cite{MR823176} and others.

Inspired by the theory of deformation quantization, Rieffel defined the notion of $\theta$-deformation as a generalization of the noncommutative torus~\cite{MR1002830}.  It takes a C$^*$-algebra with an action of $\R^n$ as the input, and the deformation parameter is given by a skewsymmetric form on $\R^n$.  He showed that the $\theta$-deformations have the same K-groups as the original algebras~\cite{MR1237992}, extending the case of the noncommutative torus by Pimsner and Voiculescu~\cite{MR587369}.

From another point of view, the noncommutative torus can also be considered as a twisted group algebra of the discrete group $\Z^2$.  In general, given any discrete group $\Gamma$ and a $\U(1)$-valued $2$-cocycle $\omega$ on $\Gamma$, one may consider the projective unitary representations of $\Gamma$ with respect to $\omega$, which leads to the notion of the (maximal) twisted group algebra $C^*_{\omega}(\Gamma)$.  Similarly, the regular $\omega$-representation of $\Gamma$ generates the C$^*$-algebra $C^*_{r, \omega}(\Gamma)$ which generalizes the usual reduced group algebra $C^*_r(\Gamma)$.  The C$^*$-algebraic properties of such algebras were extensively studied by Packer and Raeburn in the early '90s~\citelist{\cite{MR1002543}\cite{MR1066817}}.

Concerning the latter framework, Echterhoff, L\"{u}ck, Phillips, and Walters recently proved a K-theory isomorphism result for the reduced twisted group algebras when the group satisfies the Baum--Connes conjecture with the coefficients~\cite{MR2608195}.  We note that Mathai~\cite{MR2218025} also proved the K-theory invariance under twisting by such cocycles for a slightly different class of groups, building on Lafforgue's Banach algebraic approach~\cite{MR1914617} to the Baum--Connes conjecture.

The motivation of this paper is to unify the above two K-theory isomorphism results.  We thus consider C$^*$-algebras admitting coactions of the compact quantum group $C^*_r(\Gamma)$, and deform them by the $\U(1)$-valued $2$-cocycles on $\Gamma$.  This notion is equivalent to that of cross-sectional algebras of Fell bundles~\cite{MR1488064}, but the viewpoint from compact quantum group theory is also important for conceptual understanding.  The deformation of function algebras of compact groups was utilized by Wassermann in his study of ergodic actions of compact groups~\cite{MR990110}.  In the more general context of compact quantum group coactions on arbitrary operator algebras, such construction was considered by De Rijdt and Vander Vennet~\cite{MR2664313}.  The benefit of concentrating on the quantum group $C^*_r(\Gamma)$ is that, any continuous family of $\U(1)$-valued cocycles can be captured by $\R$-valued $2$-cocycles, as seen by a standard long exact sequence argument for the coefficients $\Z$, $\R$, and $\U(1)$.

Our main result (Theorem~\ref{thm:K-grp-isom-for-BC-coeff}) states that, when $\Gamma$ satisfies the Baum--Connes conjecture with coefficients and the cocycle comes from an $\R$-valued $2$-cocycle, the K-groups of the deformed algebra are isomorphic to those of the original algebra.  We also obtain an isospectral deformation of spectral triples when the `Dirac operator' is equivariant with respect to the coaction of $C^*(\Gamma)$, and the character of the deformed spectral triple can be identified (Theorem~\ref{thm:char-spc-triple-invar}) with the original one modulo the above natural isomorphism of the K-theory.

Finally, we remark that there are several similar schemes of deformation of operator algebras which do not fall into our approach.  The deformation of Fell bundles due to Abadie--Exel~\cite{MR1856986} seems to be closest to ours.  The $q$-deformation compact quantum groups are also analogues of $2$-cocycle deformation.  There is a similar K-theoretic invariance result by Neshveyev and Tuset~\cite{MR2914062} for such quantum groups and their homogeneous spaces.

\paragraph{Acknowledgments}  This paper was written during the author's stay at Institut for Matematiske Fag, K{\o}benhavns Universitet.  He would like to thank them for their support and hospitality.  He is also grateful to Ryszard Nest, Takeshi Katsura, Reiji Tomatsu, Narutaka Ozawa, and Siegfried Echterhoff for stimulating discussions and fruitful comments.

\section{Preliminaries}
\label{sec:prelim}

When $A$ and $B$ are C$^*$-algebras, $A \otimes B$ denotes their minimal tensor product unless otherwise stated.  Likewise when $H$ and K are Hilbert spaces, $H \otimes K$ denotes their tensor product as a Hilbert space.  When $X$ is a right Hilbert C$^*$-module over a C$^*$-algebra $A$, we let $\HMEnd{X}$ denote the algebra of adjointable endomorphisms of $X$.  When $H$ is a Hilbert space, we let $H \otimes X$ denote the tensor product Hilbert C$^*$-module over $A$.

When $A$ is a C$^*$-algebra, we let $\MA(A)$ denote its multiplier algebra (which can be defined as $\HMEnd{A_A}$).

The crossed products with respect to (co)actions of locally compact quantum groups on C$^*$-algebras are understood to be the reduced ones unless otherwise specified.  Our convention is that, when $\alpha$ is an action of a discrete group $\Gamma$ on a C$^*$-algebra $A \subset B(H)$, the reduced crossed product $A \rtimes \Gamma$ is the C$^*$-algebra on $\ell^2(\Gamma) \otimes H$ generated by the operators $\lambda_g \otimes \Id_H$ for $g \in \Gamma$ and the copy of $A$ represented as the operators
\[
\tilde{\alpha}(a)\colon \delta_g \otimes \xi \mapsto \delta_g \otimes \alpha_g(a) \xi \quad (g \in \Gamma, \xi \in H, a \in A).
\]

\subsection{Group cocycles}

Let $\Gamma$ be a discrete group.  When $(G, +)$ is a commutative group, a \textit{$G$-valued $2$-cocycle} $\omega$ on $\Gamma$ is a map $\omega\colon \Gamma \times \Gamma \rightarrow G$ satisfying the cocycle identity
\begin{equation}
\label{eq:cocycle-ident-omega}
\omega(g_0, g_1) + \omega(g_0 g_1, g_2) = \omega(g_1, g_2) + \omega(g_0, g_1 g_2).
\end{equation}
A $2$-cocycle $\omega$ is said to be \textit{normalized} when it satisfies
\begin{align}
  \label{eq:cocycle-normalization}
  \omega(g, e) &= \omega(e, g) = 1,&
  \omega(g, g^{-1}) &= 1
\end{align}
for any $g \in \Gamma$.  Two cocycles $\omega$ and $\omega'$ are said to be \textit{cohomologous} when there exists a map $\psi\colon \Gamma \rightarrow \U(1)$ satisfying
\begin{equation}\label{eq:cohom-cocycles}
\psi(g) \psi(h) \omega(g, h) \overline{\psi(g h)} = \omega'(g, h).
\end{equation}
Any $2$-cocycle satisfying~\eqref{eq:cocycle-ident-omega} is cohomologous to a normalized one.

In this paper we only consider the cases $G = \R$ and $G = \U(1)$ as the target group of cocycles.  Note that when $\omega$ is an $\R$-valued $2$-cocycle, we obtain a $\U(1)$-valued cocycle $e^{i \omega}$ by putting $e^{i \omega}(g, h) = e^{i \omega(g, h)}$.

When $\omega$ is a $\U(1)$-valued $2$-cocycle on $\Gamma$, \textit{the twisted reduced group C$^*$-algebra} $C^*_{r, \omega}(\Gamma)$ is the C$^*$-algebraic span of the operators $\lambda^{(\omega)}_g \in B(\ell^2 \Gamma)$ for $g \in \Gamma$ defined by
\[
\lambda^{(\omega)}_g \delta_h = \omega(g, h) \delta_{g h}.
\]

Given $\Gamma$ and $\omega$ as above, we can consider the fundamental unitary $W = \sum_g \delta_g \otimes \lambda_g$ and another unitary operator $\sum_{g, h} \omega(g, h) \delta_g \otimes \delta_h$ representing $\omega$, both represented on $\ell^2(\Gamma)^{\otimes 2}$.  Then the unitary operator
\begin{equation}
  \label{eq:reg-omega-rep-unitary}
  W^{(\omega)} = W \omega\colon \delta_h \otimes \delta_k \mapsto \beta(h, k) \delta_h \otimes \delta_{h k}
\end{equation}
in the von Neumann algebra $\ell^\infty(\Gamma) \otimes B(\ell^2(\Gamma))$ is called the \textit{regular $\omega$-representation} unitary.  The algebra $C^*_{r, \omega}(\Gamma)$ can be also defined as the C$^*$-algebraic span of the operators $\phi \otimes \iota(W^{(\omega)})$ for $\phi \in \ell^1(\Gamma) = \ell^\infty(\Gamma)_*$.

If $\omega$ is normalized, the generators $(\lambda^{(\omega)}_g)_{g \in \Gamma}$ satisfy the relations
\begin{align*}
\lambda^{(\omega)}_g \lambda^{(\omega)}_h &= \omega(g, h) \lambda^{(\omega)}_{g h}, &
(\lambda^{(\omega)}_g)^* &= \lambda^{(\omega)}_{g^{-1}}.
\end{align*}
From this formula we see that the vector state for $\delta_e$ is tracial.  This trace is called the \textit{standard trace} $\tau$ on $C^*_{r, \omega}(\Gamma)$.

When $\omega$ and $\omega'$ are cohomologous as in~\eqref{eq:cohom-cocycles}, the algebras $C^*_{r, \omega}(\Gamma)$ and $C^*_{r, \omega'}(\Gamma)$ are isomorphic via the map $\lambda^{(\omega)}_g \mapsto \overline{\psi(g)}\lambda^{(\omega')}_g$.   In the following, we always assume that $\omega$ is normalized.

We let $\overline{\omega}$ denote the complex conjugate cocycle $\overline{\omega}(g, h) = \overline{\omega(g, h)}$.  Then the twisted algebra $C^*_{r, \overline{\omega}}(\Gamma)$ is antiisomorphic to $C^*_{r, \omega}(\Gamma)$ as follows.  Using~\eqref{eq:cocycle-normalization}, we obtain the equality
\begin{equation*}
\begin{split}
\omega(h^{-1}, g^{-1}) \omega(g, h) \omega(g h, h^{-1} g^{-1}) &= \omega(h^{-1}, g^{-1}) \omega(h, h^{-1} g^{-1}) \omega(g, g^{-1})\\
 &= \omega(h, h^{-1}) \omega(g, g^{-1}) = 1,
\end{split}
\end{equation*}
which shows $\omega(h^{-1}, g^{-1}) = \overline{\omega}(g, h)$.  From this it is easy to see that the map $\lambda^{(\omega)}_g \mapsto \lambda^{(\overline{\omega})}_{g^{-1}}$ defines an antiisomorphism from $C^*_{r, \omega}(\Gamma)$ to $C^*_{r, \overline{\omega}}(\Gamma)$.

\subsection{Crossed product presentation of twisted group algebras}

The reduced group algebra $C^*_r(\Gamma)$ admits the structure of function algebra of a compact quantum group by the coproduct map $\delta(\lambda_g) = \lambda_g \otimes \lambda_g$.

Suppose that $\alpha$ and $\beta$ are $\U(1)$-valued $2$-cocycles on $\Gamma$.  Then, with the unitary regular $\beta$-representation unitary~\eqref{eq:reg-omega-rep-unitary}, we have
\[
W^{(\beta)} (\lambda^{(\alpha \cdot \beta)}_g \otimes \Id_{\ell^2(\Gamma)}) (W^{(\beta)})^* = \lambda^{(\alpha)}_g \otimes \lambda^{(\beta)}_g
\]
for any $g \in \Gamma$.  This way we obtain a C$^*$-algebra homomorphism
\[
C^*_{r, \alpha \cdot \beta}(\Gamma) \rightarrow C^*_{r, \alpha}(\Gamma) \otimes C^*_{r, \beta}(\Gamma), \quad \lambda^{(\alpha \cdot \beta)}_g \mapsto \lambda^{(\alpha)}_g \otimes \lambda^{(\beta)}_g.
\]
When either of $\alpha$ or $\beta$ is trivial, we obtain the coactions
\begin{align*}
\delta^{(\omega)}_l &\colon C^*_{r, \omega}(\Gamma) \rightarrow C^*_{r}(\Gamma) \otimes C^*_{r, \omega}(\Gamma),&
\delta^{(\omega)}_r &\colon C^*_{r, \omega}(\Gamma) \rightarrow C^*_{r,\omega}(\Gamma) \otimes C^*_{r}(\Gamma)
\end{align*}
of $C^*_r(\Gamma)$ on the twisted group algebra $C^*_{r, \omega}(\Gamma)$.  Note that these two coactions carry the same information.

The crossed product algebra $C^*_{r, \omega}(\Gamma) \rtimes_{\delta_r} C_0(\Gamma)$ with respect to the coaction $\delta^{(\omega)}_r$ is the C$^*$-algebra generated by $\delta^{(\omega)}_r(C^*_{r, \omega}(\Gamma))$ and $1 \otimes C_0(\Gamma)$ in $B(\ell^2(\Gamma) \otimes \ell^2(\Gamma))$.  This crossed product is actually isomorphic to the compact operator algebra
\[
\cptops(\ell^2(\Gamma)) \simeq \Gamma \ltimes_{\lambda} C_0(\Gamma) \simeq C^*_r(\Gamma) \rtimes_{\delta_r} C_0(\Gamma),
\]
where $\lambda$ in the middle denotes the left translation action of $\Gamma$ on $C_0(\Gamma)$.  Concretely, this isomorphism is given by the map
\begin{equation}
\label{eq:isom-tw-cross-c0-untw-cross-c0}
C^*_{r, \omega}(\Gamma) \rtimes_{\delta^{(\omega)}_r} C_0(\Gamma) \rightarrow C^*_r(\Gamma) \rtimes_{\delta_r} C_0(\Gamma),
\quad \lambda^{(\omega)}_g \delta_h \mapsto \omega(g, h) \lambda_g \delta_h.
\end{equation}
The crossed product $C^*_{r, \omega}(\Gamma) \rtimes_{\delta_r} C_0(\Gamma)$ admits the dual action $\hat{\delta}^{(\omega)}_r$ of $\Gamma$ defined by
\begin{equation*}
(\hat{\delta}^{(\omega)}_r)_k(\lambda^{(\omega)}_g \delta_h) = \lambda^{(\omega)}_g \delta_{h k^{-1}} \quad (g, h, k \in \Gamma).
\end{equation*}
If we regard $\hat{\delta}^{(\omega)}_r$ as an action of $\Gamma$ on $C^*_r(\Gamma) \rtimes_{\delta_r} C_0(\Gamma)$ via the isomorphism~\eqref{eq:isom-tw-cross-c0-untw-cross-c0}, the dual coaction can be expressed as
\begin{equation}
\label{eq:omega-delta-hat-on-untw-crossed-prod}
(\hat{\delta}^{(\omega)}_r)_k(\lambda_g \delta_h) = \overline{\omega(g, h)} \omega(g, h k^{-1}) \lambda_g \delta_{h k^{-1}}.
\end{equation}
By the Takesaki--Takai duality, the crossed product $\cptops(\ell^2(\Gamma)) \rtimes_{\hat{\delta}^{(\omega)}_r} \Gamma$ is strongly Morita equivalent to $C^*_{r, \omega}(\Gamma)$.

\subsection{Coaction of quantum groups and braided tensor products}

Suppose that $A$ is a $C^*_r(\Gamma)$-C$^*$-algebra.  Thus, $A$ admits a coaction $\alpha$ of $C^*_r(\Gamma)$ given by an injective homomorphism
\[
\alpha \colon A \rightarrow C^*_r(\Gamma) \otimes A
\]
 which satisfies the comultiplicativity $\iota \otimes \alpha \circ \alpha = \delta \otimes \iota \circ \alpha$, and the \textit{the cancellation property} (also called \textit{the continuity} of $\alpha$) meaning that $(C^*_r(\Gamma) \otimes 1) \alpha(A)$ is dense in $C^*_r(\Gamma) \otimes A$.  We write the coaction as $\alpha(x) = \sum_g \lambda_g \otimes \alpha^{(g)}(x)$.  Thus $x = \alpha^{(g)}(x)$ is equivalent to $\alpha(x) = \lambda_g \otimes x$.  Note that linear span $A_\fin$ of such elements, the \textit{elements of finite spectrum}, are dense in $A$.  This fact will be frequently utilized later to verify the images of various homomorphisms.

Suppose that $A$ is represented on a Hilbert space $H$.  Then a unitary $X$ in $\MA(C^*_r(\Gamma) \otimes \cptops(H))$ is said to be a \textit{covariant representation} for $\alpha$ if it satisfies
\[
\delta \otimes \iota (X) = X_{1 3} X_{2 3}, \quad X^* (1 \otimes a) X = \alpha(a).
\]

By analogy with the case of $\Gamma = \Z^2$~\cite{MR2738561}*{Section~3}, we would like to consider `the diagonal coaction' $\alpha \otimes \delta^{(\omega)}_l$ of $C^*_r(\Gamma)$ on $A \otimes C^*_{r, \omega}(\Gamma)$.  Nonetheless, a naive attempt
\[
A \otimes C^*_{r, \omega}(\Gamma) \rightarrow C^*_r(\Gamma) \otimes A \otimes C^*_{r, \omega}(\Gamma), \quad a \otimes x \mapsto \alpha(a)_{1 2} \delta^{(\omega)}_l(x)_{1 3}
\]
does not define an algebra homomorphism unless $\Gamma$ is commutative.  To remedy this we appeal to the notion of braided tensor ploduct which takes into account of the noncommutativity of $\Gamma$.

We consider an action $\Adj^{(\omega)}$ of $\Gamma$ on $C^*_{r, \omega}(\Gamma)$ given by
\[
\Adj^{(\omega)}_g(\lambda^{(\omega)}_h) = \lambda^{(\omega)}_g \lambda^{(\omega)}_h (\lambda^{(\omega)}_g)^* = \omega(g, h) \omega(g h, g^{-1}) \lambda^{(\omega)}_{g h g^{-1}}.
\]
Let $\tildAdj^{(\omega)}$ denote the algebra homomorphism
\[
C^*_{r, \omega}(\Gamma) \rightarrow \MA(C_0(\Gamma) \otimes C^*_{r, \omega}(\Gamma)),
\quad x \mapsto \sum_h \delta_h \otimes \Adj^{(\omega)}_{h^{-1}}(x).
\]
This is implemented as the adjoint by the $\omega$-representation unitary $W^{(\omega)}$, and satisfies $\iota \otimes \tildAdj^{(\omega)} \circ \tildAdj^{(\omega)} = \hat{\delta} \otimes \iota \circ \tildAdj^{(\omega)}$.  Hence it defines a coaction of the dual quantum group $(C_0(\Gamma), \hat{\delta})$, that is, an action of $\Gamma$.

Combined with the coaction $\delta^{(\omega)}_l$ of $C^*_r(\Gamma)$, the algebra $C^*_{r, \omega}(\Gamma)$ becomes a \textit{$\Gamma$-Yetter--Drinfeld-C$^*$-algebra}~\cite{MR2566309}.  It amounts to verifying the commutativity of the diagram
\begin{equation}
\begin{CD}
\label{eq:tw-grp-alg-YD-Gamma-alg-comm-diag}
C^*_{r, \omega}(\Gamma) @>{\delta^{(\omega)}_l}>> \hat{S} \otimes C^*_{r, \omega}(\Gamma) @>{\iota \otimes \tildAdj^{(\omega)}}>> \MA(\hat{S} \otimes S \otimes C^*_{r, \omega}(\Gamma)) \\
@VV{\tildAdj^{(\omega)}}V @. @VV{\Sigma_{12}}V \\
\MA(S \otimes C^*_{r, \omega}(\Gamma)) @>{\iota \otimes \delta^{(\omega)}_l}>> \MA(S \otimes \hat{S} \otimes C^*_{r, \omega}(\Gamma)) @>{\Adj_W}>> \MA(S \otimes \hat{S} \otimes C^*_{r, \omega}(\Gamma))
\end{CD},
\end{equation}
where $\hat{S} = C^*_r(\Gamma)$, $S = C_0(\Gamma)$, $W$ is the fundamental unitary $\sum_h \delta_h \otimes \lambda_{h}$ in $\MA(C_0(\Gamma) \otimes C^*_r(\Gamma))$, and $\Sigma$ is the transposition of tensors.  If we track the image of $\lambda^{(\omega)}_g \in C^*_{r, \omega}(\Gamma)$ along the top-right arrows, we obtain
\begin{multline*}
\lambda^{(\omega)}_g \mapsto \lambda_g \otimes \lambda^{(\omega)}_g
\mapsto \sum_h \lambda_g \otimes \delta_h \otimes (\lambda^{(\omega)}_h)^* \lambda^{(\omega)}_g \lambda^{(\omega)}_h \mapsto \sum_h \delta_h \otimes \lambda_g \otimes (\lambda^{(\omega)}_h)^* \lambda^{(\omega)}_g \lambda^{(\omega)}_h .
\end{multline*}
Similarly, if we go along the left-bottom arrows, we obtain
\begin{multline*}
\lambda^{(\omega)}_g \mapsto \sum_h \delta_h \otimes (\lambda^{(\omega)}_h)^* \lambda^{(\omega)}_g \lambda^{(\omega)}_h  \mapsto \sum_h \delta_h \otimes \lambda_{h g h^{-1}} \otimes (\lambda^{(\omega)}_h)^* \lambda^{(\omega)}_g \lambda^{(\omega)}_h \\
\mapsto \sum_h \delta_h \otimes \lambda_g \otimes (\lambda^{(\omega)}_h)^* \lambda^{(\omega)}_g \lambda^{(\omega)}_h,
\end{multline*}
where we used
\begin{equation*}
\begin{split}
\delta^{(\omega)}_l((\lambda^{(\omega)}_h)^* \lambda^{(\omega)}_g \lambda^{(\omega)}_h ) &= \omega(h^{-1}, g) \omega(h^{-1} g, h) \lambda_{h^{-1} g h} \otimes \lambda^{(\omega)}_{h^{-1} g h}\\
 &= \lambda_{h^{-1} g h} \otimes (\lambda^{(\omega)}_h)^* \lambda^{(\omega)}_g \lambda^{(\omega)}_h .
\end{split}
\end{equation*}
Combining these, we conclude that the diagram~\eqref{eq:tw-grp-alg-YD-Gamma-alg-comm-diag} is indeed commutative.

As proved in~\cite{MR2566309}*{Theorem~3.2}, a Yetter--Drinfeld algebra is the same thing as an algebra endowed with a coaction of the function algebra of the \textit{Drinfeld double} $\Double(\Gamma)$.  In our setting, $\Double(\Gamma)$ is represented by the algebra $C_0^r(\Double(\Gamma)) = C_0(\Gamma) \otimes C^*_r(\Gamma)$ endowed with the coproduct
\[
\Delta = (\Sigma \circ \Adj_W)_{2 3} \circ \hat{\delta} \otimes \delta\colon \delta_h \otimes \lambda_g \mapsto \sum_{h' h'' = h} (\delta_{h'} \otimes \lambda_{h'' g h''^{-1}}) \otimes (\delta_{h''} \otimes \lambda_g).
\]
Thus, the dual `convolution' algebra $C^*_r(\Double(\Gamma))$ can be considered as the Drinfeld double Hopf algebra of $C^*_r(\Gamma)$ and $C_0(\Gamma)$ in the C$^*$-algebraic framework.

When $\omega$ is a $\U(1)$-valued $2$-cocycle on $\Gamma$, the Yetter--Drinfeld algebra structure on $C^*_{r, \omega}(\Gamma)$ corresponds to the coaction
\[
C^*_{r, \omega}(\Gamma) \rightarrow \MA(C_0^r(\Double(\Gamma)) \otimes C^*_{r, \omega}(\Gamma)), \quad \lambda^{(\omega)}_g \mapsto \sum_h \delta_h \otimes \lambda_{h^{-1} g h} \otimes \Adj^{(\omega)}_{h^{-1}}(\lambda^{(\omega)}_g).
\]

Let $A$ be a $C^*_r(\Gamma)$-C$^*$-algebra.  The \textit{braided tensor product} $A \boxtimes C^*_{r, \omega}(\Gamma)$ of $A$ and $C^*_{r, \omega}(\Gamma)$~\cite{MR2566309}*{Definition~3.3} is the C$^*$-algebra of operators on the Hilbert C$^*$-module $\ell^2(\Gamma) \otimes A \otimes C^*_{r, \omega}(\Gamma)$ generated by the operators of the form $\alpha(a)_{1 2} \tildAdj^{(\omega)}(x)_{1 3}$ for $a \in A$ and $x \in C^*_{r, \omega}(\Gamma)$.  By means of the conditional expectation $\iota \otimes \tau$ from $A \otimes C^*_{r, \omega}(\Gamma)$ onto $A$, we may regard $A \boxtimes C^*_{r, \omega}(\Gamma)$ as a subalgebra of $\HMEnd{\ell^2(\Gamma) \otimes A \otimes \ell^2(\Gamma)}$.  Note that our convention (the Yetter--Drinfeld algebra being the second component in the braided tensor product) is different from that of~\cite{MR2566309}*{Definition~3.3}.  By~\cite{MR2182592}*{Proposition~8.3}, we have
\[
A \boxtimes C^*_{r, \omega}(\Gamma) = \overline{\alpha(A)_{1 2} \tildAdj^{(\omega)}(C^*_{r, \omega}(\Gamma))_{1 3}}
\]
as a closed linear subspace of $\HMEnd{\ell^2(\Gamma) \otimes A \otimes \ell^2(\Gamma)}$.

The braided tensor product $A \boxtimes C^*_{r, \omega}(\Gamma)$ admits a natural coaction $\alpha \otimes \delta^{(\omega)}_l$ of $C^*_r(\Gamma)$ which we shall call the \textit{diagonal coaction}.  It is given by
\[
\alpha \otimes \delta^{(\omega)}_l ( \alpha(a)_{1 2} \tildAdj^{(\omega)}(x)_{1 3}) = \delta \otimes \iota (\alpha(a))_{1 2 3} \iota \otimes \tildAdj^{(\omega)} ( \delta^{(\omega)}_l (x))_{1 2 4}.
\]

\subsection{Exterior equivalence of actions}

Let us briefly recall the notion of exterior equivalence between the (co)actions on C$^*$-algebras by the quantum groups $C^*_r(\Gamma)$ and $C_0(\Gamma)$.

Let $\alpha$ and $\beta$ be actions of $\Gamma$ on a C$^*$-algebra $A$.  These two actions are said to be \textit{exterior equivalent} when there exists a family $(u_g)_{g \in \Gamma}$ of unitaries in $\MA(A)$ satisfying $u_g \alpha_g(u_h) = u_{g h}$ and $\beta_g = \Adj_{u_g} \circ \alpha_g$ for any $g, h \in \Gamma$.  Two actions $\alpha$ and $\beta$ of $\Gamma$ on different algebras $A$ and $B$ are said to be \textit{outer conjugate} if there is an isomorphism $\phi\colon B \rightarrow A$ such that the action $(\phi \beta_g \phi^{-1})_g$ on $A$ is exterior equivalent to $\alpha$.  Outer conjugate actions define isomorphic crossed products, and the dual (co)actions on the crossed products become conjugate.

Similarly, two coactions $\alpha$ and $\beta$ of $C^*_r(\Gamma)$ on a C$^*$-algebra $A$ are said to be exterior equivalent when there is a unitary element $X$ in $C^*_r(\Gamma) \otimes A$ satisfying $X_{2 3} \iota \otimes \alpha(X) = \delta \otimes \iota (X)$ and $X \alpha(x) X^* = \beta(x)$ for $x \in A$.  Such $X$ is called an $\alpha$-cocycle.

\section{Deformation of algebras}
\label{sect:deform-alg}

\subsection{Definition and elementary properties}
\label{sec:defin-elem-prop}

\begin{dfn}
Let $A$ be a C$^*$-algebra with a coaction $\alpha$ of $C^*_r(\Gamma)$, and $\omega$ be a $\U(1)$-valued $2$-cocycle on $\Gamma$.  We define the \textit{deformation $A_{\alpha, \omega}$ of $A$ with respect to $\alpha$ and $\omega$} to be the fixed point algebra $(A \boxtimes C^*_{r, \overline{\omega}}(\Gamma))^{C^*_r(\Gamma)}$ under the diagonal coaction $\alpha \otimes \delta^{(\overline{\omega})}_l$.  When there is no source of confusion about $\alpha$, we write $A_\omega$ instead of $A_{\alpha, \omega}$ and call it \textit{the $\omega$-deformation} of $A$.
\end{dfn}

\begin{prop}
\label{prop:A-omega-alt-dfn}
 Let $\Gamma$, $\omega$, and $A$ be as above.  Then the deformed algebra $A_\omega$ is isomorphic to the subalgebra $A'_\omega$ of $C^*_{r, \omega}(\Gamma) \otimes A$ consisting of the elements $x$ satisfying $\iota \otimes \alpha(x)_{2 1 3} = \delta^{(\omega)}_l \otimes \iota(x)$.
\end{prop}

\begin{proof}
Note that the C$^*$-algebras $A \boxtimes C^*_{r, \overline{\omega}}(\Gamma)$ and $C^*_{r, \omega}(\Gamma) \otimes A \otimes C^*_{r, \overline{\omega}}(\Gamma)$ are represented on $\ell^2(\Gamma) \otimes A \otimes C^*_{r, \overline{\omega}}(\Gamma)$.  We have a homomorphism $\Phi$ from the former to the latter by $x \mapsto W^{(\overline{\omega})}_{1 3} x (W^{(\overline{\omega})})^*_{1 3}$.  The effect of $\Phi$ on the generators of $A \boxtimes C^*_{r, \overline{\omega}}(\Gamma)$ is described by
 \begin{align*}
 \alpha(x)_{1 2} &\mapsto \sum_g \lambda^{(\omega)}_g \otimes \alpha^{(g)}(x) \otimes \lambda^{(\overline{\omega})}_g,&
 \tildAdj^{(\overline{\omega})}(y)_{1 3} &\mapsto y_{3}
 \end{align*}
 for $x \in A_\fin$ and $y \in C^*_{r, \overline{\omega}}(\Gamma)$.  Thus the image of $\Phi$ is $A'_\omega \otimes C^*_{r, \overline{\omega}}(\Gamma)$, and the corresponding coaction of $C^*_r(\Gamma)$ is simply given by $(\iota \otimes \delta^{(\overline{\omega})}_l)_{2 1 3}$.  Hence the fixed point algebra is given by $A'_\omega$.
\end{proof}

The above characterization implies that we obtain the correct algebras when either of the input is `trivial'.

\begin{cor}
 When the $C^*_r(\Gamma)$-C$^*$-algebra $(A, \alpha)$ is given by the pair $(C^*_r(\Gamma), \delta)$, the deformed algebra $A_\omega$ is isomorphic to $C^*_{r, \omega}(\Gamma)$.
\end{cor}

\begin{proof}
By Proposition~\ref{prop:A-omega-alt-dfn}, we may identify the braided tensor product with the subalgebra of $C^*_{r, \omega}(\Gamma) \otimes C^*_r(\Gamma)$ spanned by $\lambda^{(\omega)}_g \otimes \lambda_g$ for $g \in \Gamma$.  As this is equal to the image of $\delta^{(\omega)}_l$, we obtain the assertion.
\end{proof}

\begin{cor}
\label{cor:triv-param-case}
 Let $A$ be a C$^*$-algebra with a coaction $\alpha$ of $C^*_r(\Gamma)$.  When the cocycle $\omega$ is trivial, the deformed algebra $A_\omega$ is isomorphic to $A$.
\end{cor}

\begin{proof}
In this case the algebra $A'_\omega$ in Proposition~\ref{prop:A-omega-alt-dfn} is the image of $\alpha$.  Hence we obtain $A_\omega \simeq A$.
\end{proof}

Also straightforward from the definition is that, when the coaction $\alpha$ is trivial, $A_\omega$ is isomorphic to $A$ for any $2$-cocycle $\omega$.

\begin{rmk}
When $a \in A_\fin$, the expression $\sum_g \lambda^{(\omega)}_g \otimes \alpha^{(g)}(a)$ defines an element in $A'_\omega$.  We let $a^{(\omega)}$ denote the corresponding element in $A_\omega$.  The $\omega$-deformation $A_\omega$ can be regarded as a certain C$^*$-algebraic completion of the vector space $\ensemble{ a^{(\omega)} \suchthat a \in A_\fin} \simeq A_\fin$ endowed with the twisted $*$-algebra structure
\begin{align*}
a^{(\omega)} b^{(\omega)} &= \sum_{g, h} \omega(g, h) (\alpha^{(g)}(a) \alpha^{(h)}(b))^{(\omega)},&
(a^{(\omega)})^* &= \sum_g (\alpha^{(g)}(a)^*)^{(\omega)}.
\end{align*}
\end{rmk}

\begin{example}
 Let $A$ be a $\T^n$-C$^*$-algebra for some $n$, and $(\theta_{j k})_{j k}$ be a skewsymmetric real matrix of size $n$.  Then the \textit{$\theta$-deformation}~\cite{MR1002830} $A_\theta$ of $A$ is given by $(A \otimes C(\T^n)_\theta)^{\T^n}$, where $C(\T^n)_\theta$ is the universal C$^*$-algebra generated by $n$ unitaries $u_1, \ldots, u_n$ satisfying $u_j u_k = e^{i \theta_{j k}} u_k u_j$, and $\T^n$ acts on $A \otimes C(\T^n)_\theta$ by the diagonal action.
 
 The algebra $C(\T^n)_\theta$ can be regarded as the twisted group algebra of $\Z^n$ with the $2$-cocycle $\omega(x, y) = e^{i (\theta x, y)}$.  By Proposition~\ref{prop:A-omega-alt-dfn}, $A_\theta$ can be identified with $A_\omega$.
\end{example}

\begin{example}
\label{exmpl:twisted-crossed-prod}
 Let $B$ be a $\Gamma$-C$^*$-algebra.  Then the reduced crossed product $\Gamma \ltimes B$ is a $C^*_r(\Gamma)$-C$^*$-algebra by the dual coaction.  If $\omega$ is a $2$-cocycle on $\Gamma$, the deformed algebra $(\Gamma \ltimes B)_\omega$ can be identified with the twisted reduced crossed product $\Gamma \ltimes_{\alpha, \omega} B$~\cite{MR0241994}.
\end{example}

There is another coaction of $C^*_r(\Gamma)$ on $A \boxtimes C^*_{r, \omega}(\Gamma)$, given by
\[
\alpha(x)_{1 2} \tildAdj^{(\overline{\omega})}(y)_{1 3} \mapsto \iota \otimes \alpha \circ \alpha(x)_{1 2 3} \tildAdj^{(\overline{\omega})}(y)_{2 4}.
\]
We denote this coaction by $\alpha_\omega$.  It is implemented as the adjoint by the dual fundamental unitary $\hat{W} = \sum_g \lambda_g \otimes \delta_g$.  It can be easily seen from the definitions that the two coactions $\alpha_\omega$ and $\alpha \otimes \delta^{(\overline{\omega})}_l$ of $C^*_r(\Gamma)$ commute with each other.  Hence $\alpha_\omega$ restricts to the fixed point subalgebra $A_\omega$ of $\alpha \otimes \delta^{(\overline{\omega})}_l$.

\begin{rmk}
\label{rmk:deformation-param-additivity}
 When $\omega$ and $\eta$ are $\U(1)$-valued $2$-cocycles on $\Gamma$, we have $(A_\omega)_\eta = A_{\omega \cdot \eta}$ for any $C^*_r(\Gamma)$-C$^*$-algebra $A$.
\end{rmk}

We have the following generalization of the isomorphism~\eqref{eq:isom-tw-cross-c0-untw-cross-c0}.

\begin{prop}
\label{prop:cross-prod-C-0-Gamma-A-omega}
 The crossed product algebra $C_0(\Gamma) \ltimes_{\alpha_\omega} A_\omega$ is isomorphic to the corresponding algebra $C_0(\Gamma) \ltimes_{\alpha} A$ of the untwisted case.
\end{prop}

\begin{proof}
 We identify $A_\omega$ with the algebra $A'_\omega$ of Proposition~\ref{prop:A-omega-alt-dfn}.  Thus, the crossed product $C_0(\Gamma) \ltimes_{\alpha_\omega} A_\omega$ is represented by the C$^*$-algebra of operators generated by $(\delta_h)_1$ and $\sum_g \lambda_g \otimes \lambda^{(\omega)}_g \otimes \alpha^{(g)}(x)$ on $\ell^2(\Gamma)^{\otimes 2} \otimes A$.
 
 Let $V$ be the unitary operator $\delta_k \otimes \delta_{k'} \mapsto \omega(k^{-1}, k') \delta_k \otimes \delta_{k'}$.  The assertion follows once we prove that the image of $\Phi = \Adj_{V_{1 2}}\colon A'_\omega \rightarrow B(\ell^2(\Gamma)^{\otimes 2} \otimes A)$ is equal to 
 \[
 C_0(\Gamma) \ltimes A' = \vee \ensemble{ (\delta_h)_1, \sum_g \lambda_g \otimes \lambda_g \otimes \alpha^{(g)}(x) \suchthat h \in \Gamma, x \in A}.
 \]
If $h \in \Gamma$ and $x \in A'_\omega$ has finite spectrum, the action of $\Phi(\alpha_\omega(x) (\delta_h)_1)$ on the vector $\delta_k \otimes \delta_{k'} \otimes b$ is given by
\[
\sum_g \delta_{h, k} \overline{\omega(k^{-1}, k')} \omega(k^{-1} g^{-1}, g k') \delta_{g h} \otimes \delta_{g k'} \otimes \alpha^{(g)}(x) b.
\]
Using the cocycle identity for $\omega$, we see that this is equal to
\[
\sum_g \omega(g, h) \delta_{h, k} \delta_{g h} \otimes \delta_{g k'} \otimes \alpha^{(g)}(x) b,
\]
which is equal to the action of $\sum_{g} \omega(g, h) \lambda_g \otimes \lambda_g \otimes \alpha^{(g)}(x) (\delta_h)_1$.  This operator is indeed in $C_0(\Gamma) \ltimes A'$.
\end{proof}

We have the following expression of $\hat{\alpha}_\omega$,
\begin{multline}
\label{eq:tw-cross-prod-YD-alg-C-0-coact}
(\hat{\alpha}_\omega)_k((\sum_g \lambda_g \otimes \lambda_g \otimes \alpha^{(g)}(x)) (\delta_h)_1) \\
= \omega(g, h k^{-1}) \overline{\omega(g, h)}(\sum_g \lambda_g \otimes \lambda_g \otimes \alpha^{(g)}(x)) (\delta_{h k^{-1}})_1) ,
\end{multline}
regarded as an action on $C_0(\Gamma) \ltimes_\alpha A$ via the isomorphism $\Phi$ in the proof of Proposition~\ref{prop:cross-prod-C-0-Gamma-A-omega}.  By the Takesaki--Takai duality, $A_\omega$ is strongly Morita equivalent to the crossed product $\Gamma \ltimes_{\hat{\alpha}_\omega} C_0(\Gamma) \ltimes_{\alpha_\omega} A_\omega$ with respect to the dual action $\hat{\alpha}_\omega$.

\subsection{Approximation property}
\label{sec:appr-prop}

For each $g \in \Gamma$, let $A_g$ denote the corresponding spectral subspace consisting of the elements $x \in A$ satisfying $\alpha^{(g)}(x) = x$.  Recall that the Fell bundle $(A_g)_{g \in \Gamma}$ associated with $A$ has the \textit{approximation property}~\cite{MR1488064}*{Definition~4.4} when there is a sequence $a_i$ of functions from $\Gamma$ into $A_e$ satisfying
\begin{gather}
\label{eq:approx-seq-bddness}
\sup_i \norm{\sum_g a_i(g)^* a_i(g)} < \infty,\\
\label{eq:approx-seq-approximation}
\lim_i \sum_h a_i(g h)^* b a_i(h) = b \quad (g \in \Gamma, b \in A_g).
\end{gather}
If $\Gamma$ is amenable, any $C^*_r(\Gamma)$-C$^*$-algebra has the approximation property.  This property also holds when $A$ is given as $\Gamma \ltimes_\beta B$ for some amenable action $\beta$ of a discrete group $\Gamma$ on a unital C$^*$-algebra $B$.

\begin{lem}
\label{lem:approx-property-invariance}
Let $\Gamma$, $\omega$, and $A$ be as above.  The Fell bundle associated with $A$ has the approximation property if and only if the one associated with $A_\omega$ has the approximation property.
\end{lem}

\begin{proof}
 The algebra $(A_\omega)_e$ is naturally isomorphic to $A_e$.  Hence we may regard $a_i$ as a sequence of functions with values in $A_\omega$.  Then the condition~\eqref{eq:approx-seq-bddness} is automatic.  The other one~\eqref{eq:approx-seq-approximation} follows from the equalities
 \begin{align*}
 \norm{b^{(\omega)}} &= \norm{b}, &
 a_i(g h)^* b^{(\omega)} a_i(h) &= (a_i(g h)^* b a_i(h))^{(\omega)}
 \end{align*}
 for any $g \in \Gamma$ and $b \in A_g$.
\end{proof}

We have the following adaptation of~\cite{MR1237992}*{Theorem~4.1} in our context.

\begin{prop}
\label{prop:nucl-invar}
 Let $\Gamma$, $\omega$, and $A$ be as above, and suppose that the Fell bundle associated with $A$ has the approximation property.  Then $A_\omega$ is nuclear if and only if $A$ is nuclear.
\end{prop}

\begin{proof}
 The Fell bundle associated with $A_\omega$ also has the approximation property by Lemma~\ref{lem:approx-property-invariance}.  By Remark~\ref{rmk:deformation-param-additivity}, it is enough to prove that $A_\omega$ is nuclear when $A$ is nuclear.
 
 By the amenability of the Fell bundle associated with $A_\omega$, the maximal and the reduced crossed product coincide for the dual action of $\Gamma$ on $C_0(\Gamma) \ltimes_{\alpha_\omega} A_\omega$~\cite{MR1895615}*{Corollary~3.6}.  Since $C_0(\Gamma) \ltimes_{\alpha_\omega} A_\omega$ is nuclear by Proposition~\ref{prop:cross-prod-C-0-Gamma-A-omega}, we conclude that its crossed product by $\Gamma$ is also nuclear, c.f.~\cite{MR1926869}*{the proof of Theorem~5.3, (2) $\Rightarrow$ (3)}.
\end{proof}

The above result can be proved by the combination of Lemma~\ref{lem:approx-property-invariance} and the nuclearity for amenable Fell bundles with nuclear unit fiber.  See Appendix for the details.

\subsection{Conjugacy of coactions}
\label{sec:conjugacy-coactions}

\begin{prop}
\label{prop:rep-of-deformed-alg-from-covar-rep}
Let $\Gamma$, $\omega$ be as in Section~\ref{sec:defin-elem-prop}, and $A$ be a $C^*_r(\Gamma)$-C$^*$-algebra represented on a Hilbert space $H$.  Suppose that there is a covariant representation $X \in \MA(C^*_r(\Gamma) \otimes \cptops(H))$ of $C^*_r(\Gamma)$.  Then, the action of $C^*_{r, \omega}(\Gamma) \otimes A$ on $\ell^2(\Gamma) \otimes H$ restricts to the one of the algebra $A'_\omega$ of Proposition~\ref{prop:A-omega-alt-dfn} on $X^* (\delta_e \otimes H)$.
\end{prop}

\begin{proof}
Recall that the dual fundamental unitary $\hat{W} = \sum_g \lambda_g^* \otimes \delta_g$ satisfies $\delta(x) = \hat{W}^* (1 \otimes x) \hat{W}$.  Hence the assumption
\[
\delta \otimes \iota(X^*) = \delta \otimes \iota(X^*)_{2 1 3} = X^*_{1 3} X^*_{2 3}
\]
implies
 \[
 X^*_{1 3} X^*_{2 3} (\delta_e \otimes \delta_e \otimes \xi) = \Adj_{\hat{W}^*_{1 2}}(X^*_{2 3}) (\delta_e \otimes \delta_e \otimes \xi) = \hat{W}^*_{1 2} X^*_{2 3} (\delta_e \otimes \delta_e \otimes \xi).
 \]
for $\xi \in H$.  Thus, any $\eta \in X^* (\delta_e \otimes H)$ satisfies $X^*_{1 3} (\delta_e \otimes \eta) = \hat{W}^*_{1 2} (\delta_e \otimes \eta)$.

Conversely, if we had $X^*_{1 3} (\delta_e \otimes \xi) = \hat{W}^*_{1 2} (\delta_e \otimes \xi)$ for some $\xi \in \ell^2(\Gamma) \otimes H$, we can write $\xi$ as $\sum_g \delta_g \otimes \xi_g$ and conclude that $X^* (\delta_e \otimes \xi_g) = \delta_g \otimes \xi_g$ for any $g$, i.e., $\xi = X^* (\delta_e \otimes \sum_g \xi_g)$.  Hence we can identify $X^*(\delta_e \otimes H)$ with the subspace $\ensemble{ \xi \suchthat X^*_{1 3} \xi = \hat{W}^*_{1 2} \xi}$ of $\delta_e \otimes \ell^2(\Gamma) \otimes H$ via the embedding $\xi \mapsto \delta_e \otimes \xi$.

By the covariance of $X$, we can characterize $A'_\omega$ as the subalgebra of $C^*_{r, \omega}(\Gamma) \otimes A$ satisfying
\[
\hat{W}^*_{1 2} ( 1 \otimes a) \hat{W}_{1 2} = X^*_{1 3} (1 \otimes a) X_{1 3}.
\]
If $\xi \in X^* (\delta_e \otimes H)$ and $a \in A'_\omega$, one has
\[
X^*_{1 3} (1 \otimes a ) (\delta_e \otimes \xi) = \hat{W}^*_{1 2} ( 1 \otimes a) \hat{W}_{1 2} X^*_{1 3} (\delta_e \otimes \xi) = \hat{W}^*_{1 2} ( 1 \otimes a) (\delta_e \otimes \xi),
\]
which proves the assertion.
\end{proof}

\begin{prop}
\label{prop:ext-equiv-str-Mor-equiv-deform}
 Let $\alpha$ and $\beta$ be exterior equivalent coactions of $C^*_r(\Gamma)$ on $A$, and $\omega$ be a $2$-cocycle on $\Gamma$.  Then the corresponding deformed algebras $A_{\alpha, \omega}$ and $A_{\beta, \omega}$ are strongly Morita equivalent.
\end{prop}

\begin{proof}
Let $U$ be an $\alpha$-cocycle satisfying $U \alpha(x) U^* = \beta(x)$.  As in the standard argument, the rank one Hilbert $A$-module $A$ admits a coaction of $C^*_r(\Gamma)$ defined by
\[
X_U \colon \xi \otimes x \mapsto \alpha(x) U^*( \xi \otimes 1) \quad (\xi \in \ell^2(\Gamma), x \in A).
\]
This coaction is covariant with respect to the coaction $\alpha$ on $A$ for the left $A$-module structure and $\beta$ for the right.  Then, as in Proposition~\ref{prop:rep-of-deformed-alg-from-covar-rep}, we can take the closed subspace $X_U (\delta_e \otimes A)$ in the Hilbert C$^*$-module $C^*_{r, \omega}(\Gamma) \otimes A$ which is closed under the left action of $A'_{\alpha, \omega}$ and the right action of $A'_{\beta, \omega}$.  This bimodule is an imprimitivity bimodule between the two algebras.
\end{proof}

\begin{cor}
\label{cor:A-omega-st-Mor-equiv-iter-twist-cross-prod}
 Let $A$ be a $C^*_r(\Gamma)$-C$^*$-algebra and $\omega$ be a $2$-cocycle on $\Gamma$.  Then the deformed algebra $A_\omega$ is strongly Morita equivalent to the twisted crossed product $\Gamma \ltimes_{\hat{\alpha}, \omega} C_0(\Gamma) \ltimes_\alpha A$.
\end{cor}

\begin{proof}
The double dual coaction of $C^*_r(\Gamma)$ on the iterated crossed product $C^*_r(\Gamma) \ltimes_{\hat{\alpha}} C_0(\Gamma) \ltimes_\alpha A$ and the amplification of $\alpha$ on $\cptops(\ell^2(\Gamma)) \otimes A$ are outer conjugate by the Takesaki--Takai duality.  The assertion follows from Proposition~\ref{prop:ext-equiv-str-Mor-equiv-deform} and the natural identification $(\cptops \otimes A)_\omega \simeq \cptops \otimes A_\omega$.
\end{proof}

This corollary shows that the twisted crossed product (Example~\ref{exmpl:twisted-crossed-prod}) is the universal example of $\omega$-deformation up to the strong Morita equivalence.  We can also see that Proposition~\ref{prop:cross-prod-C-0-Gamma-A-omega}, and the resulting strong Morita equivalence between $A_\omega$ and $\Gamma \ltimes_{\hat{\alpha}_\omega} C_0(\Gamma) \ltimes_\alpha A$ is an adaptation of the `untwisting' of twisted crossed products by Packer--Raeburn~\cite{MR1002543}*{Corollary~3.7}.

\subsection{\texorpdfstring{K}{K}-theory isomorphism of deformed algebras}
\label{sec:K-thry-isom-omega-deformations}

Let $\omega$ be a  normalized $\U(1)$-valued $2$-cocycle on $\gamma$.  For each $k \in \Gamma$, consider the unitary element
\begin{equation}
\label{eq:v-k-dfn}
v_k = \left ( \sum_g \omega(g k, k^{-1}) \delta_g \right ) \left( \sum_h \lambda_{h k^{-1} h^{-1}} \delta_h \right )
\end{equation}
in $\MA(C^*_r(\Gamma) \rtimes_{\delta_r} C_0(\Gamma)) = B(\ell^2(\Gamma))$.  The second sum is actually the unitary $\rho_k$ which implements the right translation $\delta_h \mapsto \delta_{h k^{-1}}$.  From the relation
\[
\lambda_{g'^{-1}} v_k \lambda_{g'} = \sum_g \omega(g k, k^{-1}) \delta_{g'^{-1} g} \rho_k = \sum_g \omega(g' g k, k^{-1}) \delta_g \rho_k,
\]
we conclude
\[
v_k \lambda_{g'} v_k^{-1} = \lambda_{g'} \sum_g \omega(g' g k, k^{-1}) \overline{\omega(g k, k^{-1})} \delta_g.
\]
Combining this and
\[
v_k \delta_{h'} v_k^{-1} = \rho_k \delta_{h'} \rho_k^{-1} = \delta_{h' k^{-1}},
\]
we obtain
\[
v_k \lambda_{g'} \delta_{h'} v_k^{-1} = \omega(g' h', k^{-1}) \overline{\omega(h', k^{-1})} \lambda_{g'} \delta_{h' k^{-1}}.
\]
By the cocycle condition for $\omega$ and~\eqref{eq:omega-delta-hat-on-untw-crossed-prod}, we see that the right hand side above is equal to $(\hat{\delta}^{(\omega)}_r)_{k}(\lambda_{g'} \delta_{h'})$.

The failure of the multiplicativity of $(v_k)_k$ is controlled by the cocycle $\omega$.  By 
\begin{equation*}
\begin{split}
v_{k} v_{k'} &= \sum_g \omega(g k, k^{-1}) \delta_g \rho_k \sum_h \omega(h k', k'^{-1}) \delta_h \rho_{k'}\\
&= \sum_{g = h k^{-1}} \omega(g k, k^{-1}) \omega(h k', k'^{-1}) \delta_g \rho_{k k'}
\end{split}
\end{equation*}
and the cocycle relation~\eqref{eq:cocycle-ident-omega} for $g_0 = g k k'$, $g_1 = k'^{-1}$, and $g_2 = k^{-1}$, we obtain
\begin{equation}
\label{eq:v-k-multiplicativity-with-omega}
v_k v_{k'} = \omega(k'^{-1}, k^{-1}) v_{k k'}.
\end{equation}

\begin{rmk}
\label{rmk:cocycle-untw-on-fin-subgr}
Suppose that the cocycle $\omega$ above is of the form $e^{i \omega_0}$ for some $\R$-valued $2$-cocycle $\omega_0$.  When $H$ is a finite subgroup of $\Gamma$, the $2$-cocycle $\omega_0|_H$ is a coboundary because $H^2(H, \R)$ is trivial.  Hence there exists a map $\phi$ from $H$ into $\R$ satisfying
\begin{equation}
\label{eq:tilde-omega-0-is-cbd-alpha}
\omega_0(h_0, h_1) = \phi(h_0) - \phi(h_0 h_1) + \phi(h_1) \quad (h_0, h_1 \in H).
\end{equation}
The normalization condition on $\omega_0$ implies the one $\phi(e) = 0$ for $\phi$.  The unitaries $(e^{- i \phi(h)} v_h)_{h \in H}$ are multiplicative by~\eqref{eq:v-k-multiplicativity-with-omega}, and they implement the action $\hat{\delta}^{(\omega)}_r|_H$ on $C^*_r \rtimes C_0(\Gamma)$ modulo the isomorphism~\eqref{eq:isom-tw-cross-c0-untw-cross-c0} by ~\eqref{eq:omega-delta-hat-on-untw-crossed-prod}.
\end{rmk}

Now, assume that $\omega$ is induced by an $\R$-valued $2$-cocycle $\omega_0$ as above.  Our goal is to show that the K-groups of $A_\omega$ are isomorphic to those of $A$.

Let $I$ denote the closed unit interval $[0, 1]$.  Generalizing the method of~\cite{MR2608195}*{Section~1}, we put $\omega^\theta = e^{i \theta \omega_0}$ for $\theta \in I$ and consider the following C$^*$-$C(I)$-algebra $A_{\omega^\star}$ over $I$, whose fiber at $\theta$ is given by $A_{\omega^\theta}$.

First, we consider the Hilbert space $L^2(I; \ell^2(\Gamma)) \simeq L^2(I) \otimes \ell^2(\Gamma)$, and the operators $\lambda^{(\overline{\omega}^\star)}_g$ for $g \in \Gamma$ defined by
\[
(\lambda^{(\overline{\omega}^\star)}_g \xi)_\theta = \lambda^{(\overline{\omega}^\theta)}_g \xi_\theta \quad (\xi \in L^2(I; \ell^2(\Gamma)), \theta \in I).
\]
Thus we obtain a C$^*$-$C(I)$-algebra $C^*_{r,\overline{\omega}^\star}(\Gamma)$, given as the C$^*$-algebra generated by these operators and the natural action of $C(I)$ on $L^2(I; \ell^2(\Gamma))$.

Next, we see that $C^*_{r,\overline{\omega}^\star}(\Gamma)$ is a $\Gamma$-Yetter--Drinfeld algebra by the coaction
\[
\lambda^{(\overline{\omega}^\star)}_g \mapsto \lambda_g \otimes \lambda^{(\overline{\omega}^\star)}_g, f \mapsto 1 \otimes f \quad(g \in \Gamma, f \in C(I))
\]
of $C^*_r(\Gamma)$ and the one
\[
\lambda^{(\overline{\omega}^\star)}_g \mapsto \delta_h \otimes (\lambda^{(\overline{\omega}^\star)}_h)^* \lambda^{(\overline{\omega}^\star)}_g \lambda^{(\overline{\omega}^\star)}_h, f \mapsto 1 \otimes f \quad(g, h \in \Gamma, f \in C(I))
\]
of $C_0(\Gamma)$.  This $C(I)$-algebra and its $\Gamma$-Yetter--Drinfeld algebra structure is induced by the twisted fundamental unitary
\[
W^{(\overline{\omega}^\star)} = W \overline{\omega}^\star\colon \delta_g \otimes \delta_h \mapsto (e^{- \theta i \omega_0(g, h)} \delta_g \otimes \delta_{g, h})_\theta
\]
on $L^2(I; \ell^2(\Gamma)^{\otimes 2})$ which commutes with $C(I)$.

Thus we can take the braided tensor product $A \boxtimes C^*_{r,\overline{\omega}^\star}(\Gamma)$ which is again a C$^*$-$C(I)$-algebra with a compatible coaction of $C^*_r(\Gamma)$.  The algebra $A_{\omega^\star}$ is defined to be the fixed point algebra for this coaction.  By an argument analogous to Proposition~\ref{prop:A-omega-alt-dfn}, this is isomorphic to the subalgebra $A'_{\omega^\star}$ of $A \otimes C^*_{r, \omega^\star}(\Gamma)$ consisting of the elements $a$ satisfying $\alpha_{1 3}(a) = \delta^{(\omega^\star)}_{1 2}(a)$.

The algebra $A_{\omega^\star}$ admits a coaction $\alpha_{\omega^\star}$ of $C^*_r(\Gamma)$ defined in the obvious way.  The crossed product $C_0(\Gamma) \ltimes_{\alpha_{\omega^\star}} A_{\omega^\star}$ is a $\Gamma$-C$^*$-$C(I)$-algebra, and the evaluation at each fiber is a $\Gamma$-homomorphism.

\begin{lem}
\label{lem:C-0-Gamma-crossed-prod-cont-field-triv-wo-Gamma-action}
 The C$^*$-$C(I)$-algebra $C_0(\Gamma) \ltimes_{\alpha_{\omega^\star}} A_{\omega^\star}$ is isomorphic to the constant field with fiber $C_0(\Gamma) \ltimes_{\alpha} A$.
\end{lem}

\begin{proof}
 The proof is essentially the same as that of Proposition~\ref{prop:cross-prod-C-0-Gamma-A-omega}.  The formula
 \[
 (V \delta_k \otimes \delta_{k'})_\theta = \omega^\theta(k^{-1}, k') \delta_k \otimes \delta_{k'}.
 \]
defines unitary operator $V$ on $L^2(I; \ell^2(\Gamma)^{\otimes 2})$ which commutes with $C(I)$.  When $x \in A$ and $h \in \Gamma$, the constant section $(\sum_g \lambda_g \otimes \lambda_g \otimes \alpha^{(g)}(x)) (\delta_h)_1$ of $C(I) \otimes C_0(\Gamma) \ltimes_{\alpha} A'$ is mapped to the element
 \[
  \left ( \left ( \sum_g \overline{\omega^\theta(g, h)} \lambda_g \otimes \lambda^{(\omega^\theta)}_g \otimes \alpha^{(g)}(x) \right ) (\delta_h)_1 \right )_\theta
 \]
 of $C_0(\Gamma) \ltimes_{\alpha_{\omega^\star}} A'_{\omega^\star}$.
\end{proof}

Thus, $A_{\omega^\star}$ is $\RKK(I, -, -)$-equivalent to the crossed product of $C(I) \otimes C_0(\Gamma) \ltimes_{\alpha} A$ by an action of $\Gamma$ corresponding to $\hat{\alpha}_\star$ via the isomorphism of Lemma~\ref{lem:C-0-Gamma-crossed-prod-cont-field-triv-wo-Gamma-action}.  Using~\eqref{eq:tw-cross-prod-YD-alg-C-0-coact}, we can express this action as
\begin{multline}
\label{eq:Gamma-on-cont-field}
\left ( (\hat{\alpha}_\star)_k \left( \left ( \sum_g \lambda_g \otimes \lambda_g \otimes \alpha^{(g)}(x) \right ) (\delta_h)_1 \right ) \right )_\theta \\
= \omega^\theta(g, h k^{-1}) \overline{\omega^\theta(g, h)} \left ( \sum_g \lambda_g \otimes \lambda_g \otimes \alpha^{(g)}(x) \right ) (\delta_{h k^{-1}})_1 
\end{multline}
for $x \in A$ and $h, k \in \Gamma$.

\begin{rmk}
  The C$^*$-$C(I)$-algebra $A_{\omega^\star}$ becomes a continuous field of C$^*$-algebras when the Fell bundle associated with $A$ has the approximation property~\cite{MR990592}*{Corollary~2.7}, see also the proof of Proposition~\ref{prop:nucl-invar}.
\end{rmk}

\begin{prop}
\label{prop:fin-subgrp-ext-equiv-on-A-C-Gamma-C-0}
  Let $H$ be any finite subgroup of $\Gamma$ and $\theta \in I$.  Then the restriction of $\hat{\alpha}_\star$ to $H$ is outer conjugate to the restriction of the constant field of the action $\hat{\alpha}_{\omega^\theta}$.
\end{prop}

\begin{proof}
We first prove the assertion for the case $\theta = 0$.  As in Remark~\ref{rmk:cocycle-untw-on-fin-subgr}, we can take a map $\phi$ from $H$ to $\R$ satisfying~\eqref{eq:tilde-omega-0-is-cbd-alpha}.  Now, consider the unitaries
\[
(w_k)_{\theta'} = e^{-i \theta' \phi(k)} \sum_{g} \omega^{\theta'}(g k, k^{-1}) \delta_{g} \quad (\theta' \in I)
\]
in $\MA(C(I) \otimes C_0(\Gamma))$ for $k \in H$.  This is a $\hat{\delta}_r$-cocycle.  Indeed, for any $\theta'$ we have
\begin{multline}
\label{eq:calc-w-k-cocycle-ident}
\left ( w_k (\hat{\delta}_r)_k(w_{k'}) \right )_{\theta'} \\
= e^{-i \theta' \phi(k)} \left ( \sum_{g} \omega^{\theta'}(g k, k^{-1}) \delta_{g} \right ) e^{-i \theta' \phi(k')} \left ( \sum_{g'} \omega^{\theta'}(g' k', k'^{-1}) \delta_{g' k^{-1}} \right ) \\
= e^{- i\theta'(\phi(k) + \phi(k'))} \sum_g \omega^{\theta'}(g k, k^{-1}) \omega^{\theta'}(g k k', k'^{-1}) \delta_g.
\end{multline}
Using~\eqref{eq:tilde-omega-0-is-cbd-alpha}, one sees that $e^{- i \theta'(\phi(k) + \phi(k'))}$ is equal to $e^{- \theta'\phi(k k')} \overline{\omega^{\theta'}(k'^{-1}, k^{-1})}$.  Applying~\eqref{eq:cocycle-ident-omega} for $g_0 = g k$, $g_1= k'^{-1}$, and $g_2 = k^{-1}$, we see that the right hand side of~\eqref{eq:calc-w-k-cocycle-ident} is equal to $(w_{k k'})_{\theta'}$.

If we regard $(w_k)_{k \in H}$ as elements of $\MA(C_0(\Gamma) \ltimes_{\alpha} A)$, they are $\hat{\alpha}_\star$-cocycle by definition of the dual (co)action.  We next see that they implement the conjugation between $\hat{\alpha}$ and $\hat{\alpha}_{\omega^\star}$.  Indeed, recalling that $\hat{\alpha}$ is the conjugation by $(\rho_g)_g$, we can compute
\begin{multline*}
\left ( \Adj_{(w_k)_1} \circ \hat{\alpha}_k \left ( \left ( \sum_g \lambda_g \alpha^{(g)}(x) \right )(\delta_h)_1 \right ) \right )_{\theta'} \\
= \Adj_{(v_k^{(\theta')})_1} \left ( \left (\sum_g \lambda_g \alpha^{(g)}(x) \right) (\delta_h)_1 \right) \\
= \sum_g \omega^{\theta'}(g h, k^{-1}) \overline{\omega^{\theta'}(h, k^{-1})} \lambda_{g} \delta_{h k^{-1}} \lambda_g \alpha^{(g)}(x)(\delta_h)_1
\end{multline*}
using the unitaries $(v_k^{(\theta')})_k$ defined in the same way as~\eqref{eq:v-k-dfn} but $\omega$ being replaced by $\omega^{\theta'}$.  By~\eqref{eq:Gamma-on-cont-field} and the cocycle identity for $\omega^{\theta'}$, the right hand side of the above formula is indeed equal to $\hat{\alpha}_\star$.  Thus we obtain the outer conjugacy of the actions of $H$ as in the assertion for the case $\theta = 0$.

For the general value of $\theta$, we can argue in the same way as above that the actions $\hat{\alpha}$ and $\hat{\alpha}_{\omega^\theta}$ are outer conjugate.  Thus we can compose the above outer conjugacy with the constant field of conjugacy between $\hat{\alpha}$ and $\hat{\alpha}_{\omega^\theta}$, which implies the assertion for the arbitrary value of $\theta$.
\end{proof}

We recall that the `left hand side' of the Baum--Connes conjecture with coefficients can be computed in the following way.

\begin{prop}[\cite{MR2608195}*{Proposition~1.6}]
\label{prop:BC-topol-side-cpt-reduction}
 Let $G$ be a second countable locally compact group, and $A$ and $B$ be $G$-C$^*$-algebras.  If $z \in \KK^G(A, B)$ induces an isomorphism $K^H_*(A) \rightarrow K^H_*(B)$ for any compact subgroup $H$ of $G$, the Kasparov product with $z$ induces an isomorphism from $K^\topol_*(G; A)$ to $K^\topol_*(G; B)$.
\end{prop}

\begin{thm}
\label{thm:K-grp-isom-for-BC-coeff}
 Let $\Gamma$ be a discrete group satisfying the Baum--Connes conjecture with coefficients, $A$ be a $C^*_r(\Gamma)$-C$^*$-algebra, and $\omega_0$ be an $\R$-valued $2$-cocycle on $\Gamma$.  Then the K-groups $K_i(A_\omega)$ ($i = 0, 1$) of the deformed algebra $A_\omega$ are isomorphic to $K_i(A)$ for the cocycle $\omega = e^{i \omega_0}$.
\end{thm}

\begin{proof}
It is enough to show that the evaluation map $\ev_\theta$ at $\theta \in I$ for the C$^*$-$C(I)$-algebra $\Gamma \ltimes_{\hat{\alpha}_{\omega^\star}} C_0(\Gamma) \ltimes_{\alpha_{\omega^\star}} A_{\omega^\star}$ induces an isomorphism in the K-theory for any $\theta$.

Proposition~\ref{prop:fin-subgrp-ext-equiv-on-A-C-Gamma-C-0} implies that, for any finite group $H$ of $\Gamma$, the $H$-homomorphism $\ev_\theta$ induces an isomorphism of the crossed products by $H$.  By the Green--Julg isomorphism $K^H_*(X) \simeq K_*(H \ltimes X)$ which holds for any $H$-C$^*$-algebra $X$, we obtain that $\ev_\theta$ induces an isomorphism on the $K^H$-groups.  By Proposition~\ref{prop:BC-topol-side-cpt-reduction}, $\ev_\theta$ induces an isomorphism
\[
K^\topol_*(\Gamma; C_0(\Gamma) \ltimes_{\alpha_{\omega^\star}} A_{\omega^\star}) \simeq K^\topol_*(\Gamma; C_0(\Gamma) \ltimes_{\alpha_{\omega^\theta}} A_{\omega^\theta}).
\]
By the assumption on $\Gamma$, the both hand sides are isomorphic to the K-groups of the crossed products by $\Gamma$.
\end{proof}

We have a slight variation of the above theorem for the groups satisfying the strong Baum--Connes conjecture.

\begin{thm}
\label{thm:KK-isom-for-strong-BC}
 Let $\Gamma$ be a discrete group satisfying the strong Baum--Connes conjecture, $A$ be a $C^*_r(\Gamma)$-C$^*$-algebra, and $\omega_0$ be an $\R$-valued $2$-cocycle on $\Gamma$.  Then the deformed algebra $A_\omega$ is $\KK$-equivalent to $A$ for the cocycle $\omega = e^{i \omega_0}$.
\end{thm}

\begin{proof}
Recall the following formulation of the strong Baum--Connes conjecture due to Meyer--Nest~\cite{MR2193334}.  The group $\Gamma$ satisfies the strong conjecture~\cite{MR2193334}*{Definition~9.1} if and only if the descent functor $\KK^\Gamma \rightarrow \KK, A \mapsto \Gamma \ltimes_r A$ maps weak equivalences to isomorphisms~\cite{MR2193334}*{p.~213}.
 
 The evaluation maps for the C$^*$-$C(I)$-algebra $C_0(\Gamma) \ltimes_{\alpha_{\omega^\star}} A_{\omega^\star}$ are weak equivalences by Proposition~\ref{prop:fin-subgrp-ext-equiv-on-A-C-Gamma-C-0}.  Thus, the reduced crossed products by $\Gamma$ are $\KK$-equivalent if $\Gamma$ satisfies the strong Baum--Connes conjecture.
\end{proof}

\begin{rmk}
  The assumption of Theorem~\ref{thm:KK-isom-for-strong-BC} is satisfied by the groups with the Haagerup property~\cite{MR1821144}.  Meanwhile, word hyperbolic groups satisfy the assumption of Theorem~\ref{thm:K-grp-isom-for-BC-coeff}~\cite{MR2874956}.
\end{rmk}

\begin{rmk}
  Suppose that $A$ is nuclear, the Fell bundle associated with $A$ has the approximation property, and that $\Gamma$ satisfies the strong Baum--Connes conjecture.  Then the continuous field $A_{\omega^\star}$ is an $\RKK$-fibration  in the sense of~\cite{MR2511635} by Proposition~\ref{prop:nucl-invar}, Theorem~\ref{thm:KK-isom-for-strong-BC}, and~\cite{MR2511635}*{Corollary~1.6}.
\end{rmk}

\subsection{Deformation of equivariant spectral triples}

In this section we shall see that the `equivariant Dirac operators' for a given coaction of $C^*_r(\Gamma)$ leads to an isospectral deformation on the $\omega$-deformations, and that it induces the same index map modulo the K-theory isomorphism of Theorem~\ref{thm:K-grp-isom-for-BC-coeff}.

As in Proposition~\ref{prop:rep-of-deformed-alg-from-covar-rep}, let $(A, H, X)$ be a covariant representation of a $C^*_r(\Gamma)$-C$^*$-algebra $A$ on $H$.  Suppose that $D$ is a (possibly unbounded) self-adjoint operator on $H$, and $\smooth{A}$ is a dense subalgebra of $A$ such that $a (1 + D^2)^{-1}$ is compact and $[D, a]$ is bounded for any $a \in \smooth{A}$.  Thus, $(\smooth{A}, D, H)$ is an odd spectral triple.  By abuse of notation, we let $\Id_{\ell^2(\Gamma)} \otimes D$ the closure of the operator $\xi \otimes \eta \mapsto \xi \otimes D \eta$ for $\xi \in \ell^2(\Gamma)$ and $\eta \in \dom(D)$.

We assume that $\Id_{\ell^2(\Gamma)} \otimes D$ commutes $X$ (in particular, $X$ preserves the domain of $(\Id_{\ell^2(\Gamma)} \otimes D)$) and one has $\alpha^{g}(a) \in \smooth{A}$ for any $a \in \smooth{A}$ and $g \in \Gamma$.  These conditions respectively correspond to the equivariance of the Dirac operator and the smoothness of the action.  We shall call such a spectral triple as a $C^*_r(\Gamma)$-\textit{equivariant spectral triple}.  By the equivariance of $D$, the operator $\Id_{\ell^2(\Gamma)} \otimes D$ restricts to $X^* (\delta_e \otimes H)$.  This restriction is unitarily equivalent to $D$.

Let $\smooth{A}_\fin$ denote the subalgebra of $\smooth{A}$ consisting of the elements with finite $\alpha$-spectrum.  Then the commutators of $\Id_{\ell^2(\Gamma)} \otimes D$ and $\sum_g \lambda^{(\omega)}_g \otimes \alpha^{(g)}(a) \in A'_\omega$ for $a \in \smooth{A}_\fin$ are bounded.  Thus, if we let $\smooth{A}_{\omega, \fin}$ denote the algebra generated by the $a^{(\omega)}$ for $a \in \smooth{A}_\fin$, we obtain a new spectral triple
\[
(\smooth{A}_{\omega, \fin}, \Id_{\ell^2(\Gamma)} \otimes D |_{X^* (\delta_e \otimes H)}, X^* (\delta_e \otimes H)),
\]
which is an isospectral deformation of the original triple.  By means of the unitary operator $X$ between $H$ and $X^* (\delta_e \otimes H)$, we consider this as a spectral triple represented on $H$, denoted by
\[
(\smooth{A}_{\omega, \fin}, D, H).
\]

If the original spectral triple $(A, D, H)$ is even, the above construction gives an even spectral triple over $\smooth{A}_{\omega, \fin}$ provided $X$ is compatible with the grading on $H$, that is $X \in C^*_r(\Gamma) \otimes B(H)^\even$.

\begin{example}
  Let $\ell$ be a word-length function on $\Gamma$.  Then the diagonal operator $D$ on $\ell^2(\Gamma)$ by the multiplication by $\ell$ gives an odd equivariant spectral triple over the $C^*_r(\Gamma)$-algebra $C^*_r(\Gamma)$.
\end{example}

Assume that $(A, D, H)$ is an even triple, and let $F = D \absolute{D}^{-1}$ be the phase of $D$.  The above construction of the deformed spectral triple give a Fredholm module $(F, H)$ over $A_\omega$, which is in $\KK_0(A_\omega, \C)$.  The next theorem shows that this element induces essentially the same map on the $K_0$-group if $\omega$ is a real $2$-cocycle.

\begin{thm}
\label{thm:char-spc-triple-invar}
 Suppose that $\Gamma$ satisfies the Baum--Connes conjecture with coefficients and let $\omega_0$ be an $\R$-valued $2$-cocycle on $\Gamma$.  Let $A$ be a $C^*_r(\Gamma)$-C$^*$-algebra admitting an equivariant even spectral triple $(H, D)$.  Then the even Fredholm module $(F, H)$ for $F = D \absolute{D}^{-1}$ induce the same map modulo the isomorphism given in Theorem~\ref{thm:K-grp-isom-for-BC-coeff}.
\end{thm}

\begin{proof}
 The isomorphisms of the K-groups are induced by the evaluation maps of the C$^*$-$C(I)$-algebra $A'_{\omega^\star}$.
 
 The algebra $A'_{\omega^\star}$ acts on the field of Hilbert space $X(\delta_e \otimes H) \otimes L^2(I)$ over $I$, and its elements have the bounded commutator with the self-adjoint operator
 \[
 \left ( \Id_{\ell^2(\Gamma)} \otimes F|_{X(\delta_e \otimes H)} \right ) \otimes \Id_{L^2(I)}.
 \]
 This operator defines an element of $\RKK(I; A'_{\omega^\star}, C(I))$.  It is clear from the construction that, if we specialize this element to a point $\theta \in I$, we obtain the Fredholm module $(F, H)$ on $A_{\omega^\theta}$.
\end{proof}

\begin{rmk}
There is a corresponding statement for the odd equivariant spectral triples.  It can be proved in the same way, or can be reduced to the even case by taking the graded tensor product with the standard odd spectral triple over $C^{\infty}(S^1)$ of multiplicity $2$.
\end{rmk}

\section{Concluding remarks}

\begin{rmk}
Suppose that $G$ is a compact group, $\omega$ is a $2$-cocycle on the dual $\hat{G}$ of $G$.   Wassermann~\cite{MR990110} defined a deformation $C(G)_\omega$ of $C(G)$ as in~\cite{MR990110}, endowed with the action $\lambda^\omega$ of $G$. When $G$ is commutative, this construction can be identified with ours.  More generally, we can generalize this construction to arbitrary $2$-cocycles over discrete quantum groups.

When $A$ is a $G$-C$^*$-algebra, we can define its deformation by $A_\omega = (A \otimes C(G)_\omega)^G$.  We may expect similar phenomenons in this context too, but we lack nontrivial examples in this context.  For example, there is a good reason to believe that the $\U(1)$-valued $2$-cocycles on the duals of semisimple compact Lie groups which can be smoothly perturbed to the trivial one are always induced from the dual of the maximal torus~\citelist{\cite{MR996457}\cite{MR2844801}}.  In general, suppose that $H$ is a subgroup of $G$ and $\omega$ is a cocycle in $L(H) \otimes L(H) \subset L(G) \otimes L(G)$.  Then we have the natural identification $C(G)_\omega = \Ind^G_H C(H)_\omega$ which leads to $A_\omega \simeq (\Res^G_H A)_\omega$ for any $G$-C$^*$-algebra $A$.  Hence we can reduce the computation to $\hat{H}$ which is an ordinary discrete group if $H$ is commutative.  We note that a recent work of Kasprzak~\cite{MR2736961} handles this situation.
\end{rmk}

\begin{rmk}
 The compact quantum groups $C^*_r(\Gamma)$ can be characterized as the commutative ones among the general compact quantum groups.  The arguments in Section~\ref{sec:K-thry-isom-omega-deformations} depend on this commutativity in the following way.  If $G$ is a compact group as above and $A$ is a C$^*$-algebra endowed with an action $\alpha$ of $G$, we can define the deformation of $A$ by taking the fixed point algebra $(A \otimes C(G)_\omega)^{\alpha \otimes \lambda^\omega}$.  When $G$ is commutative, this algebra is invariant under $\alpha$ (or $\lambda_\omega$) by
 \[
 \alpha_g \otimes \iota \circ \alpha_h \otimes \lambda^\omega_h = \alpha_{g h} \otimes \lambda^\omega_h = \alpha_{h g} \otimes \lambda^{\omega}_h = \alpha_h \otimes \lambda^\omega_h \circ \alpha_g \otimes \iota .
 \]
 Then we can take the crossed product $G \ltimes_\alpha (A \otimes C(G)_\omega)^{\alpha \otimes \lambda^\omega}$, which is isomorphic to the corresponding algebra for the case $\omega = 1$.  This way we can reduce the problem of $(A \otimes C(G)_\omega)^{\alpha \otimes \lambda^\omega}$ to the corresponding one for the actions of $\hat{G}$ on $G \ltimes_\alpha (A \otimes C(G))^{\alpha \otimes \lambda}$.
\end{rmk}

\begin{rmk}
 For a noncommutative compact quantum group $G$, one may consider another form of deformation of the function algebra with respect to a $2$-cocycle on the dual discrete quantum group.  Namely, if $\hat{\delta}$ is the coproduct of $C^* G$ and $\omega$ is a $2$-cocycle, $\delta_\omega = \omega \hat{\delta} \omega^{-1}$ defines another coproduct on $C^* G$.  Thus the dual Hopf algebra $H_\omega$ of $(C^* G, \delta_\omega)$ can be regarded as a deformation of $C(G)$.  Moreover, the cocycle condition for $\omega$ can be relaxed to the twist condition for some associator $\Phi$.  A result of Neshveyev--Tuset~\cite{MR2861394} for $q$-deformations of simply connected simple compact Lie groups suggests that  the K-theory of $H_\omega$ do not change if $\omega$ and the associator $\Phi_\omega$ vary continuously in an appropriate sense.
\end{rmk}

\begin{rmk}
After this work was circulated as a preprint, the K-theoretic invariance under continuous $\omega$-deformation was generalized to the setting of coactions of locally compact groups and Borel cocycles in \cite{MR3067794}.
\end{rmk}

\appendix
\section*{Appendix}

This appendix provides another proof of Proposition~\ref{prop:nucl-invar}.  The argument follows the methodology of~\cite{MR1763912}.  The author thanks Narutaka Ozawa for kindly explaining this proof.

Let $\bdl{B} = (B_g)_{g \in \Gamma}$ be a Fell bundle over $\Gamma$, and put $B = C^*_r(\bdl{B})$.  For each $g \in \Gamma$, the inclusion $j_g\colon B_g \rightarrow B$ extends to an adjointable morphism of right Hilbert $B_e$-modules.  Using this, for $x \in B$ and $g \in \Gamma$, the $g$-component $x(g)$ of $x$ is characterized as the unique element $b \in B_g$ satisfying $j_g^* x j_e a = b a$ for $a \in B_e$~\cite{MR1488064}*{Proposition~2.6}.  From this form, one sees that the map
\[
E_g\colon B \rightarrow B_g, \quad x \mapsto \hat{x}(g)
\]
is completely contractive if it is regarded as a linear transformation on $B$.

Suppose that $\bdl{B}$ has the approximation property with respect to a net $(a_i)_{i \in I}$ satisfying~\eqref{eq:approx-seq-bddness} and~\eqref{eq:approx-seq-approximation}.  We may suppose that the support $X_i$ of $a_i$ is a finite subset of $\Gamma$ for each $i \in I$~\cite{MR1488064}*{Proposition~4.5}.  In the following we assume that each $X_i$ is a finite subset of $\Gamma$ containing $e$.

For each $i \in I$, the map
\[
\Phi_i\colon x \mapsto \sum_{g, h \in \Gamma} a_i(h g)^* E_h(x) a_i(g)
\]
is a completely positive map on $B$ whose image is contained in $\oplus_{g \in X_i} B_g$.  By the summation process condition~\eqref{eq:approx-seq-approximation}, we have the pointwise convergence $\Phi_i \rightarrow \Id_B$.

\begin{prop}
Suppose that $\bdl{B}$ has the approximation property.  If $B_e$ is nuclear, then the C$^*$-algebra $B$ is also nuclear.
\end{prop}

\begin{proof}
In order to avoid confusion we write $\maxten$ for the maximal tensor product and $\mintensor$ for the minimal product in this proof.  It is enough to show that $B \maxten C$ agrees with $B \mintensor C$ for an arbitrary C$^*$-algebra $C$.

First, the complete positive map $\Phi_i$ induces the maps
\begin{align*}
\Phi_i \maxten \Id_C & \colon B \maxten C \rightarrow B \maxten C,\\
\Phi_i \mintensor \Id_C & \colon B \mintensor C \rightarrow B \mintensor C
\end{align*}
on the maximal and minimal tensor products~\cite{MR2391387}*{Theorem~3.5.3}.  Since the net $\Phi_i \maxten \Id_C$ converges to $\Id_{B \maxten C}$, it is enough to show that each $\Phi_i \maxten \Id_C$ factors through the canonical quotient map $Q \colon B \maxten C \rightarrow B \mintensor C$.

Let $\bimod{E}_i$ the right Hilbert $B_e$-module completion of $\oplus_{g \in X_i} B_g$.  Using the assumption $e \in X_i$, we obtain that $\cptops(\bimod{E}_i)$ is strongly Morita equivalent to $(\bimod{E}_i, \bimod{E}_i)_{B_e} = B_e$.  Since the nuclearity is preserved under strong Morita equivalence~(see \cite{MR2276659} and references therein), $\cptops(\bimod{E}_i)$ is also nuclear.  

Now, the formula $x \mapsto (\oplus_{g \in I} j_g)^* x j_e$ defines the map
\[
\Psi_i \colon B \rightarrow \oplus_{g \in X_i} B_g, \quad b \mapsto \oplus_{g \in X_i} E_g(b),
\]
which can be considered as a linear map from $B$ to $\cptops(\bimod{E}_i)$.  By the spatial implementation, we know that it is completely bounded.  Since the minimal tensor product is functorial for completely bounded maps~\cite{MR2391387}*{Remark~3.5.4}, we obtain a map
\[
\Psi_i \mintensor \Id_C \colon B \mintensor C \rightarrow \cptops(\bimod{E}_i) \mintensor C.
\]

Next, since the submodules $(B_g)_{g \in X_i}$ of $\ell_2 (\bdl{B})$ are mutually orthogonal, the map
\[
\Xi_i\colon \cptops(\bimod{E}_i) \rightarrow B, \quad T \mapsto \bigl ( \oplus_{g \in X_i} j_g \bigr ) T ( \oplus_{g \in X_i} j_g \bigr )^*
\]
is completely positive.  Thus, we obtain a map
\[
\Xi_i \maxten \Id_C \colon  \cptops(\bimod{E}_i) \maxten C \rightarrow B \maxten C.
\]
By the nuclearity of $\cptops(\bimod{E}_i)$, we have a natural isomorphism between $\cptops(\bimod{E}_i) \mintensor C$ and $\cptops(\bimod{E}_i) \maxten C$.  Modulo this isomorphism, we have
\[
\Phi_i \maxten \Id_C = (\Xi_i \maxten \Id_C) \circ (\Psi_i \mintensor \Id_C) \circ (\Phi_i \mintensor \Id_C) \circ Q,
\]
which proves the assertion.
\end{proof}


\end{document}